\documentclass[a4paper,12pt]{article}

\usepackage[active]{srcltx}
\usepackage{verbatim}
\usepackage{amsmath}
\usepackage{amssymb}
\usepackage{url}
\usepackage{cite}
\usepackage{enumerate}

\usepackage[margin=2cm,nohead]{geometry}

\newtheorem{theorem}{Theorem}

\newtheorem{proposition}{Proposition}

\newtheorem{definition}{Definition}

\newenvironment{proof}{{\noindent\bf Proof.}}{\hfill$\Box$\\}

\DeclareMathOperator{\MixICP}{MixICP}
\DeclareMathOperator{\MixCP}{MixCP}
\DeclareMathOperator{\ICP}{ICP}
\DeclareMathOperator{\SMixICP}{SOL-MixICP}
\DeclareMathOperator{\SMixCP}{SOL-MixCP}
\DeclareMathOperator{\SICP}{SOL-ICP}
\DeclareMathOperator{\SLCP}{SOL-LCP}
\DeclareMathOperator{\SCP}{SOL-CP}
\DeclareMathOperator{\LCP}{LCP}
\DeclareMathOperator{\CP}{CP}
\DeclareMathOperator{\C}{\mathcal C}

\DeclareMathOperator{\diag}{diag}

\newcommand{\lng}{\langle}
\newcommand{\rng}{\rangle}
\newcommand{\lf}{\left}
\newcommand{\rg}{\right}
\newcommand{\R}{\mathbb R}

\newcommand{\tp}{^\top}

\newcommand{\bs}{\left(\begin{smallmatrix}}
\newcommand{\es}{\end{smallmatrix}\right)}

\begin{document}

\title{Linear complementarity problems on extended second order cones
}

\author{S. Z. N\'emeth\\School of Mathematics, University of Birmingham\\Watson Building, Edgbaston\\Birmingham B15 2TT, United Kingdom\\email: s.nemeth@bham.ac.uk
\and L. Xiao\\School of Mathematics, University of Birmingham\\Watson Building, Edgbaston\\Birmingham B15 2TT, United Kingdom\\email: Lxx490@bham.ac.uk}
\maketitle

\begin{abstract}
    In this paper, we study the linear complementarity problems on extended second order cones. We convert a linear complementarity problem on an
    extended second order cone into a mixed complementarity problem on the non-negative orthant. We state necessary and sufficient conditions for
    a point to be a solution of the converted problem. We also present solution strategies for this problem, such as the Newton method and
    Levenberg-Marquardt algorithm. Finally, we present some numerical examples.
    \vspace{1mm}
    \noindent

    \vspace{1mm}
    \noindent

\end{abstract}

\section{Introduction}

Although research in cone complementarity problems (see the definition in the beginning of the Preliminaries) goes back a few decades only, the underlying concept of complementarity is much older, being
firstly introduced by Karush in 1939\cite{karush1939minima}. It seems that the concept of complementarity problems was first considered
by Dantzig and Cottle in a technical report \cite{dantzig1963positive}, for the non-negative orthant. In 1968, Cottle and
Dantzig\cite{cottle1968complementary}
restated the linear programming problem, the quadratic programming problem and the bimatrix game problem as a complementarity problem, which
inspired the research in this field (see \cite{mangasarian1976linear, garcia1973some, borwein1989linear,alizadeh2003second,FacchineiPang2003}).

The complementarity problem is a cross-cutting area of research which has a wide range of
applications in economics, finance and other fields. Earlier works in cone complementarity problems present the theory for a general cone and the
practical applications merely for the non-negative orthant only (similarly to the books \cite{FacchineiPang2003, MR2503647}). These are related to
equilibrium in economics, engineering, physics, finance and traffic. Examples in economics are
Walrasian price equilibrium models, price oligopoly models, Nash-Cournot production/distribution models, models of invariant capital stock,
Markov perfect equilibria, models of decentralised economy and perfect competition equilibrium, models with individual markets of production
factors. Engineering and physics applications are frictional contact problems, elastoplastic structural analysis and nonlinear
obstacle problems. An example in finance is the discretisation of the differential complementarity formulation of the Black-Scholes
models for the American options \cite{jaillet1990variational}. An application to congested traffic networks is the prediction of steady-state traffic flows. In the
recent years several applications have emerged where the complementarity problems are defined by cones essentially different from the non-negative
orthant such as positive semidefinite cones, second order cones and direct product of these cones (for mixed complementarity problems containing
linear subspaces as well). Recent applications of second order cone complementarity problems are in elastoplasticity \cite{MR2925039,MR3010551},
robust game theory \cite{MR2568432,MR2522815} and robotics \cite{MR2377478}. All these applications come from the Karush-Kuhn-Tucker conditions of
second order conic optimization problems.

N\'emeth and Zhang extended the concept of second order cone in \cite{nemeth2015extended} to the extended second order cone. Their extension
seems the most natural extension of second order cones. Sznajder showed that the extended second order cones in \cite{nemeth2015extended} are
irreducible cones (i.e., they cannot be written as a direct product of simpler cones) and calculated the Lyapunov rank of these cones
\cite{RS2016}. The applications of second order cones and the elegant way of extending them suggest that the extended second order cones will be
important from both theoretical and practical point of view. Although conic optimization problems with respect to extended second order cones can
be reformulated as conic optimization problems with respect to second order cones, we expect that for several such problems using the particular
inner structure of the second order cones provides a more efficient way of solving them than solving the transformed conic optimization problem
with respect to second order cones. Indeed, such a particular problem is the projection onto an extended second order cone which is much easier
to solve directly than solving the reformulated second order conic optimization problem \cite{FN2016}.

Until now the extended second order cones of N\'emeth and Zhang were used as a working tool only for finding the solutions of mixed
complementarity problems on general cones \cite{nemeth2015extended} and variational inequalities for cylinders whose base is a general convex set
\cite{NZ2016a}.
The applications above for second order cones show the importance of these cones and motivates considering conic optimization and complementarity
problems on extended second order cones. As another motivation we suggest the application to mean variance portfolio optimization
problems \cite{markowitz1952portfolio,roy1952safety} described in Section 3.

The paper is structured as follows: In Section 2, we illustrate the main terminology and definitions used in this paper. In Section 3 we present an application of extended second order cones to portfolio optimization  problems. In Section 4, we
introduce the notion of mixed implicit complementarity problem as an implicit complementarity problem on the direct product of a cone and a Euclidean space.
In Section 5, we reformulate the linear complementarity problem as a mixed (implicit, mixed implicit) complementarity problem on the non-negative
orthant (MixCP).

Our main result is Theorem \ref{Main_thm}, which discusses the connections between an ESOCLCP  and mixed (implicit, mixed implicit)
complementarity problems. In particular, under some mild conditions, given the definition of Fischer-Burmeister (FB) regularity and of the
stationarity of a point, we prove in Theorem \ref{SP} that a point can be the solution of a mixed complementarity problem if it satisfies
specific conditions related to FB regularity and stationarity (Theorem \ref{SP}). This theorem can be used to determine whether a point is a
solution of a mixed complementarity problem converted from ESOCLCP. In Section 6, we use Newton's method and Levenberg-Marquardt algorithm to find
the solution for the aforementioned MixCP. In Section 7, we provide an example of a linear complementarity problem on an extended second
order cone. Based on the above, we convert this linear complementarity problem into a mixed complementarity problem on the non-negative orthant,
and use the aforementioned algorithms to solve it. A solution of this mixed complementarity problem will provide a solution of the
corresponding ESOCLCP.

As a first step, in this paper, we study the linear complementarity problems on extended second order cones (ESOCLCP). We find that an ESOCLCP can
be transformed to a mixed (implicit, mixed implicit) complementarity problem on the non-negative orthant. We will give the conditions for which a
point is a solution of the reformulated MixCP problem, and in this way we provide conditions for a point to be a solution of ESOCLCP.

\section{Preliminaries}

Let $m$ be a positive integer and $F:\R^m\to\R^m$ be a mapping and $y=F(x)$. The definition of the classical complementary problem
\cite{FacchineiPangI2003}
\begin{equation*}
    x \geq 0,\quad y\geq 0, \quad and \quad \langle x, y \rangle = 0,
\end{equation*}
where $\ge$ denotes the componentwise order induced by the non-negative orthant and $\lng\cdot,\cdot\rng$ is the canonical scalar product in
$\R^m$, was later extended to more general cones $K$, as follows:
\begin{equation*}
    x\in K,\quad y\in K^*, \quad and \quad \langle x, y \rangle = 0,
\end{equation*}
where $K^*$ is the dual of $K$ \cite{MR0321540}.

Let $k,\ell,\hat\ell$ be non-negative integers such that $m=k+\ell$.

Recall the definitions of the mutually dual extended second order cone $L(k,\ell)$ and $M(k,\ell)$ in $\mathbb{R}^m\equiv\R^k\times\R^\ell$:
\begin{equation}\label{elc}
L(k,\ell) = \{(x,u) \in \mathbb{R}^k\times \mathbb{R}^\ell : x \geq \|u\|e\},
\end{equation}
\begin{equation}\label{delc}
M(k,\ell) = \{(x,u) \in \mathbb{R}^k\times \mathbb{R}^\ell : e\tp x\geq \|u\|,x\ge0\},
\end{equation}
where  $e=(1, \dots, 1)\tp \in \mathbb{R}^k $. If there is no ambiguity about the dimensions, then we simply denote $L(k,\ell)$ and $M(k,\ell)$ by
$L$ and $M$, respectively.

Denote by $\lng\cdot,\cdot\rng$ the canonical scalar product in $\R^m$ and by $\|\cdot\|$ the corresponding Euclidean norm. The notation
$x\perp y$ means that $\lng x,y\rng=0$, where $x,y\in\R^m$.

Let $K\subset\R^m$ be a nonempty closed convex cone and $K^*$ its dual.

\begin{definition}
The set \[\C(K):=\left\{(x,y)\in K\times K^*:x\perp y\right\}\] is called the \emph{complementarity set} of $K$.
\end{definition}

\begin{definition}\label{cp-def}
	Let $F:\R^m\to\R^m$. Then, the complementarity problem
	$\CP(F,K)$ is defined by:

	\begin{equation}\label{cp-eq}
		\CP(F,K):\textrm{ }(x,F(x))\in\C(K).
	\end{equation}
	The solution set of $\CP(F,K)$ is denoted by $\SCP(F,K)$:
    \begin{equation*}
        \SCP(F,K) =  \{x\in \R^m: (x, F(x)) \in \C(K)\}.
    \end{equation*}
	If $T$ is a matrix, $r\in\R^m$ and $F$ is defined by $F(x)=Tx+r$, then $\CP(F,K)$ is denoted by
	$\LCP(T,r,K)$ and is called \emph{linear
	complementarity problem}. The solution set of $\LCP(T,r,K)$ is denoted by $\SLCP(T,r,K)$.
\end{definition}

\begin{definition}\label{icp-def}
	Let $G,F:\R^m\to\R^m$. Then, the implicit complementarity problem
	$\ICP(F,G,K)$ is defined by
	\begin{equation}\label{icp-eq}
		\ICP(F,G,K):\textrm{ }(G(x),F(x))\in\C(K).
	\end{equation}
	The solution set of $\ICP(F,G,K)$ is denoted by $\SICP(F,G,K)$:
    \begin{equation*}
        \SICP(F,G,K) = \{x\in \R^m:  (G(x), F(x)) \in \C(K)\}.
    \end{equation*}
\end{definition}

Let $m,k,\ell$ be non-negative integers such that $m=k+\ell$, $\Lambda\in\R^k$ be a nonempty closed convex cone and
$K=\Lambda\times\R^\ell$. Denote by $\Lambda^*$ the dual of $\Lambda$ in $\R^k$ and by $K^*$ the dual of $K$ in
$\R^k\times\R^\ell$. It is easy to check that $K^*=\Lambda^*\times\{0\}$.

\begin{definition}\label{mixcp-def}
	Consider the mappings $F_1:\R^k\times\R^\ell\to \R^k$ and $F_2:\R^k\times\R^\ell\to \R^{\hat\ell}$. The mixed
	complementarity problem
	$\MixCP(F_1,F_2,\Lambda)$ is defined by
	\begin{gather}\label{mixcp-eq}
		\MixCP(F_1,F_2,\Lambda):\left\{
		\begin{array}{l}
			F_2(x,u)=0\\\\
			(x,F_1(x,u))\in\C(\Lambda).
		\end{array}
		\right.
	\end{gather}
	The solution set of $\MixCP(F_1,F_2,\Lambda)$ is denoted by $\SMixCP(F_1,F_2,\Lambda)$:
    \begin{align*}
        \SMixCP(F_1,F_2,\Lambda) =\{x\in \R^m: F_2(x,u)=0,(x, F_1(x,u)) \in \C(\Lambda)\}.
    \end{align*}
\end{definition}

\begin{definition}\label{s0_matrix} \textrm{\cite[Definition 3.7.29]{FacchineiPang2003}}
    A matrix $\Pi \in \R^{n\times n}$ is said to be an $S_0$ matrix if the system of linear inequalities
    \begin{equation*}
        \Pi x \ge 0,\quad 0\ne x\ge 0
    \end{equation*}
    has a solution.
\end{definition}

The proof of our next result follows immediately from $K^*= \Lambda^*\times\{0\}$ and the definitions of $\CP(F,K)$ and $\MixCP(F_1,F_2,\Lambda)$.

\begin{proposition}
	Consider the mappings \[F_1:\R^k\times\R^\ell\to \R^k,\quad F_2:\R^k\times\R^\ell\to \R^\ell.\] Define the
	mapping
	\[F:\R^k\times\R^\ell\to\R^k\times\R^\ell\] by \[F(x,u)=(F_1(x,u),F_2(x,u)).\] Then,
	\[(x,u)\in\SCP(F,K)\iff (x,u)\in\SMixCP(F_1,F_2,\Lambda).\]
\end{proposition}

\begin{definition}\label{Schur_comp}\textrm{\cite[Schur complement]{zhang2006schur}}
	The notation of the Schur complement for a matrix $\Pi=\left(\begin{smallmatrix} P & Q\\R & S\end{smallmatrix}\right)$, with $P$
		nonsingular, is
    \begin{equation*}
        \left( \Pi/P \right) = S - RP^{-1}Q.
    \end{equation*}
\end{definition}

\begin{definition}\label{Lipschitz}\textrm{\cite[Definition 4.6.2]{sohrab2003basic}}
    \item[(i)]Let $I$ be an open subset with $ I \subset \R^m$ and $ f: I \rightarrow \R^m$. We say that $f$ is Lipschitz function, if there is a
	    constant $\lambda >0$ such that
    \begin{equation}
        \| f(x) - f(x') \| \le \lambda\|x - x'\| \quad \forall x, x' \in I.
    \end{equation}
    \item[(ii)] We say that $f$ is locally Lipschitz if for every $ x \in I$, there exists $ \varepsilon > 0 $ such that $f$ is Lipschitz on
	    $I\cap B_{\varepsilon}(x)$, where $B_{\varepsilon}(x)=\{y\in\R^m:\|y-x\|\le\varepsilon\}$.
\end{definition}

\section{An Application of Extended Second Order Cones to Portfolio Optimisation Problems}

Consider the following Portfolio Optimisation Problem:




\[
\min_{w}\lf\{w^{\top}\Sigma w:\textrm{ }r^{\top}w \ge R,\textrm{ }e\tp w=1\rg\},
\]

where $\Sigma\in\mathbb R^{n\times n}$ is the covariance matrix, $e=(1, \dots, 1)^{\top} \in \mathbb{R}^n$, $w\in\R^n$ is the weight of asset
allocation for the portfolio and $R$ is the required return of the portfolio.

In order to guarantee the diversified allocation of the fund into different assets in the market, a new constraint can be reasonably introduced:
\(\|w\| \le \xi,\)
where $ \xi$ is the limitation of the concentration of the fund allocation. If short selling is allowed, then $w$ can be less than zero. The introduction of this constraint can guarantee that the fund will be allocated into few assets only.

Since the covariance matrix $\Sigma$ can be decomposed into $\Sigma = U^{\top}U$, the problem can be rewritten as
\[
\min_{w,\xi,y}\lf\{y:\textrm{ }r^{\top}w \ge R,\textrm{ }\|Uw\|\le y,\textrm{ }\|w\| \le \xi,\textrm{ }e\tp w = 1\rg\}.
\]

The constraint $\|Uw\|\le y$ is a relaxation of the constraint $\|U\|\|w\|\le y$,
where $\|U\|=\max_{\|x\|\le 1}{\|Ux\|}$. The strengthened problem will become:
\[
\min_{w,\xi,y}\lf\{y:\textrm{ }r^{\top}w \ge R,\textrm{ }\|w\|e \le \left(\xi, \dfrac{y}{\|U\|}\right)^{\top},\textrm{ }e\tp w=1\rg\}.
\]
The minimal value of the objective of the original problem is at most as large as the minimal value of the objective for this latter problem.
The second constraint of the latter portfolio optimisation problem means that the point \(\left(\xi, y/\|U\|,w\right)^{\top}\) belongs
to the extended second order cone $L(2,n)$. Hence, the strenghtened problem is a conic optimisation problem with respect to an extended second order cone.

\section{Mixed Implicit Complementarity Problems}

Let $m,k,\ell,\hat\ell$ be non-negative integers such that $m=k+\ell$, $\Lambda\in\R^k$ be a nonempty, closed, convex cone and $K=\Lambda\times\R^\ell$. Denote by $\Lambda^*$
the dual of $\Lambda$ in $\R^k$ and by $K^*$ the dual of $K$ in $\R^k\times\R^\ell$.

\begin{definition}\label{mixicp-def}
	Consider the mappings \[F_1,G_1:\R^k\times\R^\ell\to \R^k,\quad F_2:\R^k\times\R^\ell\to \R^{\hat\ell}.\] The mixed implicit
	complementarity problem
	$\MixICP(F_1,F_2,G_1,\Lambda)$ is defined by
	\begin{gather}\label{mixicp-eq}
		\MixICP(F_1,F_2,G_1,\Lambda):\left\{
		\begin{array}{l}
			F_2(x,u)=0\\\\
			(G_1(x,u),F_1(x,u))\in\C(\Lambda).
		\end{array}
		\right.
	\end{gather}
	The solution set of the mixed complementarity problem $\MixICP(F_1,F_2,G_1,\Lambda)$ is denoted by $\SMixICP(F_1,F_2,G_1,\Lambda)$:
    \begin{align*}
        \SMixICP&( F_1,F_2,G_1,\Lambda)  = \\
        &\{x\in \R^m: F_2(x,u)=0,(G_1(x,u), F_1(x,u)) \in \C(\Lambda)\}.
    \end{align*}
\end{definition}

The proof of our next result follows immediately from $K^*=\Lambda^*\times \{0\}$ and the definitions of $\ICP(F,G,K)$ and $\MixICP(F_1,F_2,G_1,\Lambda)$.

\begin{proposition}
	Consider the mappings \(F_1,G_1:\R^k\times\R^\ell\to \R^k,\) \(F_2,G_2:\R^k\times\R^\ell\to \R^\ell.\) Define the mappings
	\(F,G:\R^k\times\R^\ell\to\R^k\times\R^\ell\) by \(F(x,u)=(F_1(x,u),F_2(x,u)),\) \(G(x,u)=(G_1(x,u),G_2(x,u)),\) respectively.
	Then,
	\[(x,u)\in\SICP(F,G,K)\iff (x,u)\in\SMixICP(F_1,F_2,G_1,\Lambda).\]
\end{proposition}

\section{Main Results}
The linear complementarity problem is the dual problem of a quadratic optimisation problem, which has a wide range of applications in various
areas. One of the most famous application is the portfolio optimisation problem first introduced by Markowitz \cite{markowitz1952portfolio}; see
the application of the extended second order cone to this problem presented in the Introduction.

\begin{proposition}\label{cs-esoc-prop}
	Let $x,y\in\R^k$ and $u,v\in\R^\ell\setminus\{0\}$.
	\begin{enumerate}
		\item[(i)] $(x,0,y,v)\in\C(L)$ if and only if $e\tp y\ge\|v\|$ and $(x,y)\in\C(\R^k_+)$.
		\item[(ii)] $(x,u,y,0)\in\C(L)$ if and only if $x\ge\|u\|$ and $(x,y)\in\C(\R^k_+)$.
		\item[(iii)] $(x,u,y,v):=((x,u),(y,v))\in\C(L)$ if and only if there exists a $\lambda>0$ such that $v=-\lambda u$,
			$e\tp y=\|v\|$ and $(x-\|u\|e,y)\in C(\R^k_+)$.
	\end{enumerate}
\end{proposition}

\begin{proof}
	Items (i) and (ii) are easy consequence of the definitions of $L$, $M$ and the complementarity set of
	a nonempty closed convex cone.
	\vspace{2mm}

	Item (iii) follows from Proposition 1 of \cite{FN2016}. For the sake of completeness, we will reproduce its
    proof here. First assume that there exists $\lambda>0$ such that $v=-\lambda u$,
	$e\tp y=\|v\|$ and $(x-\|u\|e,y)\in C(\R^p_+)$. Thus, $(x,u)\in L$ and $(y,v)\in M$. On the other
	hand, \[\lng (x,u),(y,v)\rng=x\tp y+u\tp v=\|u\|e\tp y-\lambda\|u\|^2=\|u\|\|v\|-\lambda\|u\|^2=0.\]
	Thus, $(x,u,y,v)\in C(L)$.

	Conversely, if $(x,u,y,v)\in C(L)$, then $(x,u)\in L$, $(y,v)\in M$ and
	\[0=\lng (x,u),(y,v)\rng=x\tp y+u\tp v\ge\|u\|e\tp y+
	u\tp v\ge\|u\|\|v\|+u\tp v\ge0.\]
	This implies the existence of a $\lambda>0$ such that $v=-\lambda u$, $e\tp y=\|v\|$ and
	$(x-\|u\|e)\tp y=0$. It follows that $(x-\|u\|e,y)\in C(\R^p_+)$.
\end{proof}

\begin{theorem}\label{Main_thm}
	Denote $z=(x,u)$, $\hat z=(x-\|u\|,u)$, $\tilde{z}=(x-t,u,t)$ and $r=(p,q)$ with $x,p\in\R^k$, $u,q\in\R^\ell$ and $t\in \R$. Let $T=\bs A & B\\C & D \es$ with $A\in\R^{k\times k}$,
	$B\in\R^{k\times\ell}$, $C\in\R^{\ell\times k}$ and $D\in\R^{\ell\times\ell}$. The square matrices $T$, $A$ and $D$ are assumed to be
	nonsingular.
	\begin{enumerate}
		\item[(i)] Suppose $u=0$. We have
        \begin{align*}
            z\in\SLCP&(T,r,L)
            \\&\iff x\in\SLCP(A,p,\R^k_+)\mbox{ and }e\tp(Ax+p)\ge\|Cx+q\|.
        \end{align*}			
		\item[(ii)] Suppose $Cx+Du+q=0$. Then, \[z\in\SLCP(T,r,L)\iff z\in\SMixCP(F_1,F_2,\R^k_+)
			\mbox{ and }x\ge\|u\|,\] where $F_1(x,u)=Ax+Bu+p$ and $F_2(x,u)=0$.
		\item[(iii)] Suppose $u\ne 0$ and $Cx+Du+q\ne 0$.
			We have \[z\in\SLCP(T,r,L)\iff z\in\SMixICP(F_1,F_2,G_1,\R^k_+),\] where
			\[F_2(x,u)=\lf(\|u\|C+ue\tp A\rg)x+ue\tp(Bu+p)+\|u\|(Du+q),\] $G_1(x,u)=x-\|u\|e$ and $F_1(x,u)=Ax+Bu+p$.
		\item[(iv)] Suppose $u\ne 0$ and $Cx+Du+q\ne 0$.
			We have \[z\in\SLCP(T,r,L)\iff \hat z\in\SMixCP(F_1,F_2,\R^k_+),\] where
			\[F_2(x,u)=\lf(\|u\|C+ue\tp A\rg)(x+\|u\|e)+ue\tp(Bu+p)+\|u\|(Du+q)\] and $F_1(x,u)=A(x+\|u\|e)+Bu+p$.
		\item[(v)] Suppose $u\ne 0$, $Cx+Du+q\ne 0$ and $\|u\|C+u\tp e A$ is a nonsingular matrix. We have
			\[z\in\SLCP(T,r,L)\iff \hat z\in\SICP(F_1,F_2,\R^k_+),\] where
			\[F_1(u)=A\lf(\lf(\|u\|C+ue\tp A\rg)^{-1}\lf(ue\tp(Bu+p)+\|u\|(Du+q)\rg)\rg)+Bu+p\] and
			\[F_2(u)=\lf(\|u\|C+ue\tp A\rg)^{-1}\lf(ue\tp(Bu+p)+\|u\|(Du+q)\rg).\]
        \item[(vi)] Suppose $u\ne 0$, $Cx+Du+q\ne 0$. We have
            \[z\in\SLCP(T,r,L)\iff \exists t>0\] such that \[\tilde{z}\in\MixCP(\widetilde{F}_1,\widetilde{F}_2,\R^k_+),\] where
            \[\widetilde{F}_1(x,u,t)=A(x+te)+Bu+p\] and
            \begin{equation}\label{mathcal_F}
                \widetilde{F}_2(x,u,t) =
                \begin{pmatrix}
                    & \lf(tC+ue\tp A\rg)(x+te)+ue\tp(Bu+p)+t(Du+q) \\
                    & t^2 - \|u\|^2
                \end{pmatrix}.
            \end{equation}

	\end{enumerate}
\end{theorem}

\begin{proof}
	\begin{enumerate}
		\item[(i)] We have that $z\in\SLCP(T,r,L)$ is equivalent to
			$(x,0,Ax+p,Cx+q)\in\C(L)$ or, by item (i) of Proposition \ref{cs-esoc-prop}, to
			$(x,Ax+p)\in\C(\R^k_+)$ and $e\tp(Ax+p)\ge\|Cx+q\|$.
		\item[(ii)] We have that $z\in\SLCP(T,r,L)$ is equivalent to
			$(x,u,Ax+Bu+p,0)\in\C(L)$ or, by item (ii) of Proposition \ref{cs-esoc-prop}, to
			$(x,Ax+Bu+p)\in\C(\R^k_+)$ and $x\ge\|u\|$, or to
			\[z\in\SMixCP(F_1,F_2,\R^k_+)\mbox{ and }x\ge\|u\|,\] where $F_1(x,u)=Ax+Bu+p$ and
			$F_2(x,u)=0$.
		\item[(iii)] Suppose that $z\in\SLCP(T,r,L)$. Then, $(x,u,y,v)\in\C(L)$, where $y=Ax+Bu+p$ and $v=Cx+Du+q$. Then, by item (iii)
			of Proposition \ref{cs-esoc-prop}, we have that $\exists\lambda>0$ such that
			\begin{equation}\label{parall-eq}
				Cx+Du+q=v=-\lambda u,
			\end{equation}
			\begin{equation}\label{prod-eq}
				e\tp (Ax+Bu+p)=e\tp y=\|v\|=\|Cx+Du+q\|=\lambda\|u\|,
			\end{equation}
			\begin{equation}\label{cpset-eq}
				(G_1(x,u),F_1(x,u))=(x-\|u\|e,Ax+Bu+p)=(x-\|u\|e,y)\in\C(\R^k_+).
			\end{equation}
			From equation \eqref{parall-eq} we obtain $\|u\|(Cx+Du+q)=-\lambda\|u\| u$, which by equation \eqref{prod-eq} implies
			$\|u\|(Cx+Du+q)=-ue\tp (Ax+Bu+p)$, which after some algebra gives
			\begin{equation}\label{zero-eq}
				F_2(x,u)=0.
			\end{equation}
			From equations \eqref{cpset-eq} and \eqref{zero-eq} we obtain that $z\in\SMixICP(F_1,F_2,G_1)$.
			\medskip
			
			Conversely, suppose that $z\in\SMixICP(F_1,F_2,G_1)$. Then,
			\begin{equation}\label{zero-eq2}
				\|u\|v+ue\tp y=\|u\|(Cx+Du+q)+ue\tp (Ax+Bu+p)=F_2(x,u)=0
			\end{equation}
			and
			\begin{equation}\label{cpset-eq2}
				(x-\|u\|e,y)=(x-\|u\|e,Ax+Bu+p)=(G_1(x,u),F_1(x,u))\in\C(\R^k_+),
			\end{equation}
			where $v=Cx+Du+q$ and $y=Ax+Bu+p$. Equations \eqref{cpset-eq2} and \eqref{zero-eq2} imply
			\begin{equation}\label{parall-eq2}
				v=-\lambda u,
			\end{equation}
			where
			\begin{equation}\label{lambda-eq}
				\lambda=(e\tp y)/\|u\|>0.
			\end{equation}
			Equations \eqref{parall-eq2} and \eqref{lambda-eq} imply
			\begin{equation}\label{norm-v-eq}
				e\tp y=\|v\|.
			\end{equation}
			By item (iii) of Proposition \ref{cs-esoc-prop}, equations \eqref{parall-eq2}, \eqref{norm-v-eq} and \eqref{cpset-eq2}
			imply \[(x,y,u,v)\in C(L)\] and therefore $z\in\SLCP(T,r,L)$.
		\item[(iv)] It is a simple reformulation of item (iii) by using the change of variables \[(x,u)\mapsto (x-\|u\|e,u).\]
		\item[(v)] Again it is a simple reformulation of item (iv) by using that $\|u\|C+u\tp e A$ is a nonsingular matrix.
        \item[(vi)]  Suppose that $z\in\SLCP(T,r,L)$. Then, $(x,u,y,v)\in\C(L)$, where $y=Ax+Bu+p$ and $v=Cx+Du+q$. Let $t = \|u\|$, Then, by item (iii)
			of Proposition \ref{cs-esoc-prop}, we have that $\exists\lambda>0$ such that
			\begin{equation}\label{parall-eq3}
				Cx+Du+q=v=-\lambda u,
			\end{equation}
			\begin{equation}\label{prod-eq3}
				e\tp (Ax+Bu+p)=e\tp y=\|v\|=\|Cx+Du+q\|=\lambda t,
			\end{equation}
			\begin{equation}\label{cpset-eq3}
				(\tilde{z},\widetilde{F}_1(x,u,t))=(x-te,Ax+Bu+p)=(x-te,y)\in\C(\R^k_+),
			\end{equation}
			where $\tilde{z} = (x-t,u,t)$. From equation \eqref{parall-eq3}, we get $t(Cx+Du+q)=-t \lambda u$, which, by equation \eqref{prod-eq3}, implies
			$t(Cx+Du+q)=-ue\tp (Ax+Bu+p)$, which after some algebra gives
			\begin{equation}\label{zero-eq3}
				\widetilde{F}_2(x,u,t)=0.
			\end{equation}
			From equations \eqref{cpset-eq3} and \eqref{zero-eq3} we obtain that $z\in\SMixCP(\widetilde{F}_1,\widetilde{F}_2,\R^k_+)$.
			\medskip
	\end{enumerate}
	$\,$ 
\end{proof}
Note that the item(vi) makes $\widetilde{F}_1(x,u,t)$ and $\widetilde{F}_2(x,u,t)$ become smooth functions by adding the variable $t$. The
smooth functions therefore make the smooth Newton's method applicable to the mixed complementarity problem.

The conversion of $\LCP$ on extended second order cones to a $\MixCP$ problem defined on the non-negative orthant is useful, because it can be
studied by using the Fischer-Burmeister function. In order to ensure the existence of the
solution of $\MixCP$, we introduce the scalar \emph{Fischer-Burmeister C-function} (see \cite{fischer1992special, fischer1995newton}).

\begin{equation*}
    \psi_{FB}(a,b) = \sqrt{a^2+b^2} - (a+b) \quad \forall (a,b) \in \mathbb{R}^2.
\end{equation*}

Obviously, $\psi_{FB}^2(a,b)$ is a continuously differentiable function on $\mathbb{R}^2$.
The equivalent FB-based equation formulation for the $\MixCP$ problem is:

\begin{equation}\label{FBform}
    0= \mathbb{F}^{\MixCP}_{FB}(x,u,t) =
    \begin{pmatrix}
         \psi(x_1,\widetilde{F}_1^1(x,u,t)) \\
         \vdots \\
         \psi(x_k,\widetilde{F}_1^k(x,u,t)) \\
         \widetilde{F}_2(x,u,t)
    \end{pmatrix},
\end{equation}
with the associated merit function:

\begin{equation*}
    \theta^{\MixCP}_{FB}(x,u,t)=\frac{1}{2}\mathbb{F}^{\MixCP}_{FB}(x,u,t)^T \mathbb{F}^{\MixCP}_{FB}(x,u,t).
\end{equation*}
We continue by calculating the Jacobian matrix for the associated merit function. If $i \in (1,...,k)$ is such that $(z_i,\widetilde{F}_1^i) \neq
(0,0)$, then the differential with respect to $z = (x,u,t)\in \R^{m+1}$ is
\begin{align*}
    \frac{\partial\left(\mathbb{F}^{\MixCP}_{FB}\right)_i}{\partial z} =
    \left(\frac{x_i}{\sqrt{x_i^2+\left(\widetilde{F}_1^i(x,u,t)\right)^2}}-1\right)e^i &\\
    +\left(\frac{\widetilde{F}_1^i(x,u,t)}{\sqrt{x_i^2+\left(\widetilde{F}_1^i(x,u,t)\right)^2}}-1\right) & \frac{\partial\widetilde{F}_1^i(x,u,t)}{\partial z},
\end{align*}		
where $e^i$ denotes the $i$-th canonical unit vector. The differential with respect to  $z_j$ with $j\neq i$ is
\begin{equation*}
    \frac{\partial\left(\mathbb{F}^{\MixCP}_{FB}\right)_i}{\partial z_j} = \left(\frac{\widetilde{F}_1^i(x,u,t)}{\sqrt{x_i^2+\left(\widetilde{F}_1^i(x,u,t)\right)^2}}-1\right)\frac{\partial\widetilde{F}_1^i(x,u,t)}{\partial z_j},
\end{equation*}
Obviously, the differential with respect to $z_j$ with $j > k$, is equal to zero. Note that if $(z_i,\widetilde{F}_1^i) = (0,0)$, then
$\frac{\partial\left(\mathbb{F}^{\MixCP}_{FB}\right)_i}{\partial z}$ will be a generalised gradient of a composite function, i.e., a closed unit
ball $B(0,1)$. However, this case will not occur in our paper.
As for the term $\widetilde{F}_2(x,u,t)$ with $i \in ( k+1,...m+1)$, the Jacobian matrix is much more simple, since
\begin{equation*}
     \frac{\partial\left(\mathbb{F}^{\MixCP}_{FB}\right)_i}{\partial z}= \frac{\partial\widetilde{F}_2^i(x,u,t)}{\partial z}.
\end{equation*}
Therefore, the Jacobian matrix for the associated merit function is:
\begin{equation*}
    \mathcal{A} =
        \begin{pmatrix}
            D_a+D_bJ_x\widetilde{F}_1(x,u,t) &&&&& D_bJ_{(u,t)}\widetilde{F}_1(x,u,t)\\
            J_x\widetilde{F}_2(x,u,t) &&&&& J_{(u,t)}\widetilde{F}_2(x,u,t)
        \end{pmatrix},
\end{equation*}
where
\begin{align*}
    &D_a= \diag
    \begin{pmatrix}
        \frac{x_i}{\sqrt{x_i^2 + \widetilde{F}_1^i(x,u,t)^2}} - 1
    \end{pmatrix},
    \qquad
    D_b=\diag
    \begin{pmatrix}
        \frac{\widetilde{F}_1^i(x,u,t)}{\sqrt{x_i^2 + \widetilde{F}_1^i(x,u,t)^2}} - 1
    \end{pmatrix},\\
    & i=1, \dots , k.
\end{align*}

Define the following index sets:
\begin{equation*}
    \begin{array}{lcl}
    \C \equiv \lf\{i: x_i \ge 0, \widetilde{F}_1^i\ge 0, x_i \widetilde{F}_1^i(x,u,t) = 0\rg\} &  & complementarity\;index \\
    \mathcal{R} \equiv \lf\{1, \dots, k\rg\}\setminus \C &  & residual\;index \\
    \mathcal{P} \equiv \lf\{i\in \mathbb{R} : x_i > 0 , \widetilde{F}_1^i(x,u,t) > 0\rg\} &  & positive\;index \\
    \mathcal{N} \equiv \mathcal{R}\setminus\mathcal{P} &  & negative\;index \\
    \end{array}
\end{equation*}

\begin{definition}\label{FB_regular}

    A point $(x,u,t) \in \R^{m+1} $ is called FB-regular for the merit function $ \theta^{\MixCP}_{FB}$ (or for the $ \MixCP\left( \widetilde{F}_1,\widetilde{F}_2,\R^k_+ \right)$) if its
    partial Jacobian matrix of $ \mathbb{F}^{\MixCP}_{FB}(x,u,t) $ with respect to x, $ J_x\widetilde{F}_1(x,u,t) $ is nonsingular and if for $\forall w \in \R^k, w\neq 0 $ with
    \begin{equation*}
       w_{\C}=0, \quad w_{\mathcal{P}}>0,  \quad w_{\mathcal{N}}<0,
    \end{equation*}
    there exists a nonzero vector $ v \in \R^k $ such that
    \begin{equation}\label{z_ineq}
        v_{\C}=0, \quad v_{\mathcal{P}}\ge 0,  \quad v_{\mathcal{N}}\le0,
    \end{equation}
    and
    \begin{equation}\label{sch_c}
        w^T\left( \Pi(x,u,t)/J_x\widetilde{F}_1(x,u,t)\right)v \ge 0,
    \end{equation}
    where
    \begin{equation*}
        \Pi(x,u,t) \equiv
        \begin{pmatrix}
             J_x\widetilde{F}_1(x,u,t)  &&&&&  J_{(u,t)}\widetilde{F}_1(x,u,t)\\
             J_x\widetilde{F}_2(x,u,t)  &&&&&  J_{(u,t)}\widetilde{F}_2(x,u,t)
        \end{pmatrix}
        \in \R^{(m+1)\times(m+1)},
    \end{equation*}
    and $ \Pi(x,u,t)/J_x\widetilde{F}_1(x,u,t) $ is the Schur complement of $ J_x\widetilde{F}_1(x,u,t)$ in $ \Pi(x,u,t) $.
\end{definition}
 In our case, for the $ \MixCP\left( \widetilde{F}_1,\widetilde{F}_2,\R^k_+ \right)$,  the Jacobian matrices are:

\begin{align*}
    J\widetilde{F}_1(x,u,t) \equiv
    \left(\widetilde{A}\textrm{ }\textrm{ }\widetilde{B}\right)
\end{align*}
and

\begin{align*}
    J\widetilde{F}_2 (x,u,t) \equiv
    \left(\widetilde{C}\textrm{ }\textrm{ }\widetilde{D}\right)
\end{align*}
where
\begin{equation*}
    \widetilde{A} = A, \qquad
    \widetilde{B} =
    \left(B\textrm{ }\textrm{ }Ae\right)
    \qquad
        \widetilde{C} =
    \begin{pmatrix}
        tC+ue^\top A \\
        0
    \end{pmatrix},
\end{equation*}

\begin{small}
\begin{equation*}
    \widetilde{D} =
    \begin{pmatrix}
    e^\top \left(A(x+te)+Bu+p\right)I + \diag(e^{\top}Bu) + tD &&&&& Cx+2tCe+ue^\top Ae+Du \\
    -2u^\top &&&&& 2t
    \end{pmatrix}.
\end{equation*}
\end{small}

In our case, if the Jacobian matrix block $ J_x\widetilde{F}_1(x,u,t) = A$ is nonsingular, then the Schur complement $
\Pi(x,u,t)/J_x\widetilde{F}_1(x,u,t) $ is

\begin{equation}\label{sch_eq}
    \left( \Pi(x,u,t)/J_x\widetilde{F}_1(x,u,t)\right) =
    \widetilde{D}-\widetilde{C}\widetilde{A}^{-1}\widetilde{B}.
\end{equation}

\begin{proposition}
	If the matrices $\widetilde{A}$ and $\widetilde{D}$ are nonsingular for any $z \in \R^{m+1}$, then the Jacobian matrix $\mathcal{A}$ for the associated merit function is nonsingular.
\end{proposition}

\begin{proof}
    It is easy to check that
    \begin{equation*}
        \mathcal{A} =
        \begin{pmatrix}
            D_a+D_b\widetilde{A} &&&&& D_b\widetilde{B}\\
            \widetilde{C} &&&&& \widetilde{D}
        \end{pmatrix}.
    \end{equation*}
    $\mathcal{A}$ is a nonsingular matrix if and only if the sub-matrix $D_a+D_b\widetilde{A}$ and its Schur complement are nonsingular, and they are nonsingular if and only if the matrices $\widetilde A$ and $\widetilde{D}$ are nonsingular.
\end{proof}

The following theorem is \cite[Theorem 9.4.4]{FacchineiPang2003}. For the sake of completeness, we provide a proof here.

\begin{theorem}\label{SP}
          A point $(x, u, t)\in \R^{m+1}$ is a solution of the $ \MixCP(\widetilde{F}_1, \widetilde{F}_2,\R^k)$ if and only if $(x,u,t) $ is an FB regular point of $ \theta_{FB}^{\MixCP}$ and $ (x, u, t)$ is a stationary point of $ \mathbb{F}_{FB}^{\MixCP}$.
\end{theorem}

\begin{proof}

    Suppose that $z^{*}=(x^{*},u^{*},t^{*})\in \SMixCP(\widetilde{F}_1,\widetilde{F}_2,\R^k)$. Then, it follows that $z^{*}$ is a global minimum
    and hence a stationary point of $ \theta_{FB}^{\MixCP}$. Thus, $(x^{*},\widetilde{F}_1(z^{*}))\in\C(\R^k_+)$, and we have
    $\mathcal{P}=\mathcal{N}=\emptyset$. Therefore, the FB regularity of $x^{*}$ holds  since $x^{*}=x_{\C}$, because there is no nonzero vector $x$ satisfying conditions (\ref{z_ineq}).
    Conversely, suppose that $x^{*}$ is FB regular and $ z^*=(x^*, u^*, t^*)$ is a stationary point of $ \theta_{FB}^{\MixCP}$. It follows that $\nabla \theta_{FB}^{\MixCP} = 0$, i.e.:

    \begin{equation*}
        \mathcal{A}^\top \mathbb{F}_{FB}^{\MixCP} =
        \begin{pmatrix}
            D_a+D_bJ_x\widetilde{F}_1(z^*) & J_x\widetilde{F}_2(z^*)\\
            D_bJ_{(u,t)}\widetilde{F}_1(z^*) & J_{(u,t)}\widetilde{F}_2(z^*)
        \end{pmatrix}
        \mathbb{F}_{FB}^{\MixCP}=0,
    \end{equation*}
    where
    \begin{align*}
        &D_a= \diag
        \begin{pmatrix}
            \frac{x_i^*}{\sqrt{(x_i^*)^2 + \widetilde{F}_1^i(z^*)^2}} - 1
        \end{pmatrix},
        \qquad
        D_b=\diag
        \begin{pmatrix}
            \frac{\widetilde{F}_1^i(z^*)}{\sqrt{(x_i^*)^2 + \widetilde{F}_1^i(z^*)^2}} - 1
        \end{pmatrix},\\
        & i=1, \dots , k.
    \end{align*}
    Hence, for any $w \in \R^{m+1}$, we have

    \begin{equation}\label{p_thm3_0}
        w^{\top}
        \begin{pmatrix}
            D_a+D_bJ_x\widetilde{F}_1(z^*) & J_x\widetilde{F}_2(z^*)\\
            D_bJ_{(u,t)}\widetilde{F}_1(z^*) & J_{(u,t)}\widetilde{F}_2(z^*)
        \end{pmatrix}
        \mathbb{F}_{FB}^{\MixCP}=0.
    \end{equation}
    Assume that $z^{*}$ is not a solution of $\MixCP $. Then, we have that the index set $\mathcal{R}$ is not empty.
    Define $v\equiv D_b\mathbb{F}_{FB}^{\MixCP}$. We have

    \begin{equation*}
        v_{\mathcal{C}} =0,
        \qquad
        v_{\mathcal{P}} >0,
        \qquad
        v_{\mathcal{N}} <0.
    \end{equation*}

    Take $w$ with

    \begin{equation*}
        w_{\mathcal{C}} =0,
        \qquad
        w_{\mathcal{P}} >0,
        \qquad
        w_{\mathcal{N}} <0.
    \end{equation*}
    From the definition of $D_a$ and $D_b$, we know that $D_a\mathbb{F}_{FB}^{\MixCP}$ and $D_b\mathbb{F}_{FB}^{\MixCP}$ have the same sign.
    Therefore,

    \begin{equation}\label{p_thm3_1}
        w^{\top}(D_a\mathbb{F}_{FB}^{\MixCP}) = w^{\top}_{\mathcal{C}}(D_a\mathbb{F}_{FB}^{\MixCP})_{\mathcal{C}} + w^{\top}_{\mathcal{P}}(D_a\mathbb{F}_{FB}^{\MixCP})_{\mathcal{P}} +w^{\top}_{\mathcal{N}}(D_a\mathbb{F}_{FB}^{\MixCP})_{\mathcal{N}} > 0.
    \end{equation}
    By the regularity of $J\widetilde{F}_1(z)^{\top}$, we have

    \begin{equation}\label{p_thm3_2}
        w^{\top}J\widetilde{F}_1(z)^{\top}(D_a\mathbb{F}_{FB}^{\MixCP})= w^{\top} J\widetilde{F}_1(z)^{\top}w \ge 0.
    \end{equation}
    The inequalities (\ref{p_thm3_1}) and (\ref{p_thm3_2}) together contradict condition (\ref{p_thm3_0}). Hence $ \mathcal{R} = \emptyset$. It means that $z^{*}$ is a solution of $\MixCP(\widetilde{F}_1,\widetilde{F}_2,\R^k)$.

\end{proof}

\section{Algorithms}
For solving a complementarity problem, there are many different algorithms available. The common algorithms include numerical methods for systems
of nonlinear equations (such as Newton's method \cite{atkinson2008introduction}), the interior point method (Karmarkar's Algorithm
\cite{karmarkar1984new}), the projection iterative method\cite{mangasarian1977solution}, and the multi-splitting method \cite{o1985multi}.
In the previous sections, we have already provided sufficient conditions for using FB regularity and stationarity to identify a solution of the
$\MixCP$ problem. In this section, we are trying to find a solution of $\LCP$ by finding the solution of $\MixCP$ which is converted from $\LCP$.
One convenient way to do this is using the Newton's Method as follows:

\begin{flushleft}
\textbf{Algorithm} (Newton's method):
\end{flushleft}

Given initial data $z^0 \in \R^{m+1}$, and $r = 10^{-7}$.

\textbf{Step 1}: Set $k = 0$.

\textbf{Step 2}: If $\mathbb{F}^{\MixCP}_{FB}(z^k)\leq r$, then STOP.

\textbf{Step 3}: Find a direction $d^k \in \R^{m+1}$ such that
\begin{equation*}
    \mathbb{F}^{\MixCP}_{FB}(z^k) + \mathcal{A}^{\top}(z^k)d^k = 0.
\end{equation*}

\textbf{Step 4}: Set $z^{k+1} := z^k + d^k $ and $k := k + 1$, go to Step 2.
\paragraph{}

 If the Jacobian matrix $\mathcal{A}^{\top}$ is nonsingular, then the direction $d^k \in \R^{m+1}$ for each step can be found. The following
 theorem, which is based on an idea similar to the one used in \cite{luenberger2015linear}, proves that such a Newton's Method can efficiently
 solve the $\LCP$ on extended second order cone (i.e. solve the problem within polynomial time), by finding the solution of the $\MixCP$:

\begin{theorem}\label{newton}
    Suppose that the Jacobian matrix $\mathcal{A}$ is nonsingular. Then, Newton's method for $\MixCP(\widetilde{F}_1,\widetilde{F}_2,\R^k_+)$ converges at least quadratically to
    \begin{equation*}
    z^* \in \SMixCP(\widetilde{F}_1,\widetilde{F}_2,\R^k_+),
    \end{equation*}
    if it starts with initial data $z^0$ sufficiently close to $z^*$.
\end{theorem}

\begin{proof}
    Suppose that the starting point $z^0$ is close to the solution $z^*$, and suppose that $\mathcal{A}$ is a Lipschitz function. There are $\rho >0, \beta_1 >0, \beta_2 >0$, such that for all $z$ with $\lVert z - z^* \rVert < \rho$, there holds $\lVert\mathcal{A}^{-1}(z)\rVert < \beta_1$, and $\lVert\mathcal{A}(z^k) - \mathcal{A} \left(z^*)\right)\rVert \leq \beta_2\lVert z^k - z^*\rVert$. By the definition of the Newton's method, we have
    \begin{align*}
        \lVert z^{k+1} - z^* \rVert & = \lVert z^k - z^* - \mathcal{A}^{-1}(z^k)\mathbb{F}^{\MixCP}_{FB}(z^k)\rVert  \\
        & = \mathcal{A}^{-1}(z^k)\left[ \mathcal{A}(z^k)(z^k - z^*) - \left(\mathbb{F}^{\MixCP}_{FB}(z^k) - \mathbb{F}^{\MixCP}_{FB}(z^*)\right)\right],
    \end{align*}
    because $\mathbb{F}^{\MixCP}_{FB}(z^*)=0$ when $z^* \in \SMixCP$. By Taylor's theorem, we have
    \begin{equation*}
        \mathbb{F}^{\MixCP}_{FB}(z^k) - \mathbb{F}^{\MixCP}_{FB}(z^*) = \int_{0}^{1} \mathcal{A}\left(z^{k} + s(z^* - z^k)\right)(x^k - z^*)ds,
    \end{equation*}
    so
    \begin{align*}
        \lVert \mathcal{A}(z^k)(z^k - z^*) & - \left(\mathbb{F}^{\MixCP}_{FB}(z^k) - \mathbb{F}^{\MixCP}_{FB}(z^*)\right)\rVert \\
         & = \left\lVert \int_{0}^{1}\left[ \mathcal{A}(z^k) - \mathcal{A} \left( z^k + s(z^* - z^k)\right)\right]  ds(z^k - z^*) \right\rVert \\
         & \leq \int_{0}^{1}\lVert\mathcal{A}(z^k) - \mathcal{A} \left( z^k + s(z^* - z^k)\right)\rVert ds\lVert z^k - z^*\rVert \\
         & \leq \lVert z^k - z^*\rVert^2 \int_{0}^{1}\beta_2 sds = \frac{1}{2}\beta_2 \lVert z^k - z^*\rVert ^2.
    \end{align*}
    Also, we have $\lVert z - z^* \rVert < \rho$, that is,
    \begin{equation*}
        \lVert z^{k+1} - z^* \rVert \leq \frac{1}{2}\beta_1 \beta_2 \lVert z^{k} - z^*\rVert^2.
    \end{equation*}
\end{proof}

Another widely-used algorithm is presented by Levenberg and Marquardt in \cite{marquardt1963algorithm}. Levenberg-Marquardt algorithm can
approach second-order convergence speed without requiring the Jacobian matrix to be nonsingular. We can approximate the Hessian matrix by:

\begin{equation*}
    \mathcal{H}(z)=\mathcal{A}^{\top}(z)\mathcal{A}(z),
\end{equation*}
and the gradient by:
\begin{equation*}
    \mathcal{G}(z)=\mathcal{A}^{\top}(z)\mathbb{F}^{\MixCP}_{FB}(z).
\end{equation*}
Hence, the upgrade step will be
\begin{equation*}
    z^{k+1} = z^k - \left[\mathcal{A}^{\top}(z^k)\mathcal{A}(z^k) +\mu \mathbb{I}\right]^{-1}\mathcal{A}^{\top}(z^k)\mathbb{F}^{\MixCP}_{FB}(z^k).
\end{equation*}

As we can see, Levenberg-Marquardt algorithm is a quasi-Newton's method for an unconstrained problem. When $\mu$ equals to zero, the step
upgrade is just the Newton's method using approximated Hessian matrix. The number of iterations of Levenberg-Marquardt algorithm to find
a solution is higher than that of Newton's method, but it works for singular Jacobian as well. The greater the parameter $\mu$, the slower
the calculation speed becomes. Levenberg-Marquardt algorithm is provided as follows:

\begin{flushleft}
\textbf{Algorithm} (Levenberg-Marquardt):
\end{flushleft}

Given initial data $z^0 \in \R^{m+1}$, $\mu = 0.005$, and $r = 10^{-7}$.

\textbf{Step 1}: Set $k = 0$.

\textbf{Step 2}: If $\mathbb{F}^{\MixCP}_{FB}(z^k)\leq r$, stop.

\textbf{Step 3}: Find a direction $d^k \in \R^{m+1}$ such that
\begin{equation*}
     \mathcal{A}(z^k)^{\top}\mathbb{F}^{\MixCP}_{FB}(z^k) + \left[\mathcal{A}^{\top}(z^k)\mathcal{A}(z^k) +\mu \mathbb{I}\right]d^k = 0.
\end{equation*}

\textbf{Step 4}: Set $z^{k+1} := z^k + d^k $ and $k := k + 1$, go to Step 2.
\paragraph{}

\begin{theorem}\label{LevenbergM} \emph{\cite{yamashita2001rate}}
    Without the nonsingularity assumption on the Jacobian matrix $\mathcal{A}$, Levenberg-Marquardt Algorithm for
    $\MixCP(\widetilde{F}_1,\widetilde{F}_2,\R^k_+)$ converges at least quadratically to
    \begin{equation*}
    z^* \in \SMixCP(\widetilde{F}_1,\widetilde{F}_2,\R^k_+),
    \end{equation*}
    if it starts with initial data $z^0$ sufficiently close to $z^*$.
\end{theorem}

The proof is omitted.

\section{A Numerical Example}

In this section, we will provide a numerical example for $\LCP$ on extended second order cones. Let $L(3,2)$ be an extended second order cone
defined by (\ref{elc}). Following the notation in Theorem \ref{Main_thm}, let $z=(x,u)$, $\hat z=(x-\|u\|,u)$, $\tilde{z}=(x-t,u,t)$ and
$r=(p,q) = \left((-55,-26,50)^{\top}, (-19,-26)^{\top} \right)$ with $x,p\in\R^3$ , $u,q\in\R^2$, and $t\in \R$. Consider
\begin{equation*}
    T=\bs A & B\\C & D \es =
    \lf(
    \begin{array}{rrrrr}
     26 &  15 &   3 &   51 & -42  \\
     -7 & -39 & -16 & -17 &  18  \\
     32 &  23 &  40 & -38 &  46  \\
      6 & -22 & -28 & -17 &  27  \\
    -38 & -25 &  24 &  47 & -16
    \end{array}
    \rg),
\end{equation*}
with $A\in\R^{3\times 3}$, $B\in\R^{3\times 2}$, $C\in\R^{2\times 3}$ and $D\in\R^{2\times 2}$. It is easy to show that square matrices T, A and D
are nonsingular. By item (vi) of Theorem \ref{Main_thm}, we can reformulate this $\LCP$ problem as a smooth $\MixCP$ problem. We will use the
Levenberg-Marquardt algorithm to find the solution of the FB-based equation formulation (\ref{FBform}) of $\MixCP$ problem. The
convergence point is:
\begin{align*}
    \tilde{z}^* & = (x - t,u,t) \\
    & = \left( \left(0, \frac{439}{660}, 0\right)^{\top}, \left(\frac{341}{1460},\frac{724}{2683}\right)^{\top}, \frac{1271}{3582} \right) .
\end{align*}
We need to check the FB regularity of $\tilde{z}^*$. It is easy to show that the partial Jacobian matrix of $\widetilde{F}_1(\tilde{z}^*)$
\begin{align*}
    J_x\widetilde{F}_1(\tilde{z}^*)  =
    \widetilde{A} =
    \lf(
    \begin{array}{rrr}
         26 &  15 &   3  \\
         -7 & -39 & -16  \\
         32 &  23 &  40
    \end{array}
    \rg)
\end{align*}
is nonsingular. Moreover, we have that
\begin{equation*}
    x - t = \left(0, \frac{439}{660}, 0\right)^{\top} \ge 0,  \qquad \widetilde{F}_1(\tilde{z}^*) = \left(\frac{3626}{145}, 0, \frac{12148}{185}\right)^{\top} \ge 0,
\end{equation*}
and therefore
\begin{equation*}
    \left<x - t, \widetilde{F}_1(\tilde{z}^*)\right> = 0.
\end{equation*}
That is, $(x, \widetilde{F}_1(\tilde{z}^*))\in \C(\R^3_+)$, so the index sets $\mathcal{P}=\mathcal{N}=\emptyset$. The matrix $\widetilde{A} $ is
invertible. In addition, we can calculate that the Schur complement of  $\Pi(\widetilde{z}^*) $ with respect to $ J_x\widetilde{F}_1(\tilde{z}^*)$:
\begin{equation*}
        \left( \Pi(\tilde{z}^*)/J_x\widetilde{F}_1(\tilde{z}^*)\right) = \widetilde{D} - \widetilde{C}\widetilde{A}^{-1}\widetilde{B}
         =
	 \lf(
	 \begin{array}{rrr}
            \frac{3991}{58}  & \frac{11387}{95}  & -\frac{7203}{268} \\
            \frac{15910}{93} & \frac{5185}{163}  & -\frac{5941}{248} \\
            -\frac{341}{740} & -\frac{741}{1373} & \frac{1271}{1791}
        \end{array}
	\rg).
\end{equation*}
The FB regularity of $x^{*}$ holds as there is no nonzero vector $x$ satisfying conditions (\ref{z_ineq}). Then, we compute the gradient of the
merit function, which is
\begin{align*}
    \mathcal{A}^\top \mathbb{F}_{FB}^{\MixCP} & =
    \begin{pmatrix}
        D_a+D_bJ_x\widetilde{F}_1(\tilde{z}^*) & J_{x}\widetilde{F}_2(\tilde{z}^*)\\
        D_bJ_{(u,t)}\widetilde{F}_1(\tilde{z}^*) & J_{(u,t)}\widetilde{F}_2(\tilde{z}^*)
    \end{pmatrix}
    \mathbb{F}_{FB}^{\MixCP} \\
    & =
    \lf(
    \begin{array}{cccccc}
         -\frac{598}{605} &  7 &               0 &  \frac{4844}{349} &  \frac{345}{1238} &                 0 \\
        -\frac{32}{21195} & 39 &               0 & -\frac{3946}{491} & -\frac{4031}{441} &                 0 \\
                        0 & 16 & -\frac{413}{415} &     -\frac{26}{7} &  \frac{1754}{111} &                 0 \\
        -\frac{33}{12610} &  7 &               0 & \frac{12462}{139} & \frac{78767}{701} &  -\frac{341}{740} \\
        -\frac{32}{21195} & 39 &               0 & \frac{13790}{131} &  \frac{9451}{105} & -\frac{741}{1373} \\
         0                & 16 &               0 & -\frac{3341}{135} & -\frac{3233}{190} & \frac{1271}{1791}
    \end{array}
    \rg)
    \begin{pmatrix}
    0 \\
    0 \\
    0 \\
    0 \\
    0 \\
    0 \\
    \end{pmatrix}
    =0.
\end{align*}
Hence, $z^*$ is a stationary point of $F^{\MixCP}_{FB}$. By Theorem \ref{SP}, we conclude that  $z^*$ is the solution of the $\MixCP$ problem. By the item (vi) of Theorem \ref{Main_thm}, we have that
\begin{align*}
    z & = (x,u) \\
    & =\left(\left(\frac{1271}{3582}, \frac{1072}{1051}, \frac{1271}{3582}\right)^{\top}, \left(\frac{341}{1480},\frac{724}{2683}\right)^{\top}\right),
\end{align*}
is the solution of $\LCP(T,r,L)$ problem.

\section{Conclusions}

In this paper, we studied the method of solving a linear complementarity problem on an extended second order cone. By checking the stationarity and
FB regularity of a point, we can verify whether it is a solution of the mixed complementarity problem. Such conversion of a linear
complementarity problem to a mixed complementarity problem reduces the complexity of the original problem. The connection between a linear
complementarity problem on an extended second order cone and a mixed complementarity problem on a non-negative orthant will be useful for our
further research about applications to practical problems, such us portfolio selection and signal processing problems.

%

\bibliographystyle{unsrt}{
\bibliography{cp-esoc}

\end{document}